\RequirePackage{fix-cm}
\documentclass[smallextended]{svjour3}       
\smartqed  
\usepackage{graphicx,lineno,hyperref,amsmath,booktabs,algorithm,multirow,amssymb}
\begin{document}

\title{A New Global Optimization Method Based on Simplex Branching for Solving a Class of Non-Convex QCQP Problems
}


\author{Bo Zhang  \and YueLin Gao   \and  Xia Liu \and XiaoLi Huang
}


\institute{ *Corresponding Author:  Bo Zhang \at
             School of Civil and Hydraulic Engineering, Ningxia University, Yinchuan 750021, PR China;
                School of Mathematics and Statistics, Ningxia University, Yinchuan, 750021, PR China\\
                  Tel.: +86-187-0960-2837\\
                \email{zbsdx121@163.com}
           \and
          YueLin Gao \at
              Ningxia province key laboratory of intelligent information and data processing, North Minzu University, Yinchuan, 750021, PR China; School of mathematics and information science, North Minzu University, Yinchuan, 750021, PR China\\
               \email{gaoyuelin@263.net}
                   \and
               Xia Liu \at
                School of Mathematics and Statistics, Ningxia University, Yinchuan, 750021, PR China\\
                \email{lingxiaoyu911@163.com} \and
               XiaoLi Huang \at
              School of mathematics and information science, North Minzu University, Yinchuan, 750021, PR China\\
               \email{hxl1569501@163.com}}

\date{Received: date / Accepted: date}

\maketitle

\begin{abstract}
Quadratic constrained quadratic programming problems often occur in various fields such as engineering practice, management science, and network communication.
This article mainly studies a non convex quadratic programming problem with convex quadratic constraints. Firstly, based on our existing results, the problem is reconstructed as an equivalent problem with a simple concave quadratic objective function in the result space, with a convex feasible domain. A global optimization algorithm for solving equivalent problems is proposed based on a branch and bound framework that can ensure the global optimality of the solution. This algorithm combines effective relaxation processes with branching processes related to new external approximation techniques. Finally, the theoretical feasibility of the algorithm was analyzed.
\end{abstract}

\section{Introduction}
Consider the following class of quadratically constrained quadratic programming (QCQP) problems  defined as
$$
({\rm QCQP})\left\{
\begin{aligned}
&\min ~\phi(x)=x^TQx+q^Tx\\
&~{\rm s.t.}~~x\in \mathcal{X},
\end{aligned}
\right.$$
where
$\mathcal{X}=\left\{x\in \mathbb{R}^n|Ax\leq b,x\geq 0,x^TQ_ix+q_i^Tx\leq d_i,i=1,2,\cdots,p\right\}$
is assumed as a bounded closed set, and we also assume the Slater condition
on $\mathcal{X}$, i.e., there is at least one interior point here within  $\mathcal{X}$;
$Q\in \mathbb{R}^{ n\times n}$ is an indefinite symmetric matrix with $r(r\leq n)$ negative eigenvalues,
$Q_i(i=1,2,\cdots,p)\in \mathbb{R}^{n\times n}$ are positive semidefinite matrices, $q\in \mathbb{R}^n$, $q_i\in \mathbb{R}^n$, $d_i\in \mathbb{R}(i=1,2,\cdots,p)$, $A\in \mathbb{R}^{m\times n}$ and $b\in \mathbb{R}^{m}$; $T$ represents the transpose of a vector (or matrix)(for example, $q_i^T$ represents the transpose of a vector $q_i$), $\mathbb{R}^n$ represents the set of all $n$-dimensional real vectors.

Problem QCQP is a common non-convex programming problem, which is widely used in practical fields such as economics, applied science, portfolio analysis, support vector machine, structural analysis, optimal control and engineering.
For a more general introduction to QCQP, we can refer to  \cite{Floud1995217,Gould2002149,Pardalos1991115} and the references cited therein.
Moreover, many other nonlinear optimization problems are special cases of such problems or can be transformed into this form, such as  celis-dennis-tapia problem \cite{Bomze2015459}, generalized trust domain problem \cite{Fortin2004041,Jeyakumar2014147}, standard quadratic programming problem \cite{Bomze2008115,Bomze2000301},  box constrained quadratic programming problem \cite{Vandenbussche2005531}, bilinear programming problem \cite{Vaish2010303} and linear multiplicative
programming (LMP) problem \cite{ShenP2020453}.
For QCQP, if the objective function is also convex, the problem is polynomial solvable with a given accuracy (see \cite{Lobo1998193}).
Nevertheless, QCQP is known to be NP-hard because its objective function is non-convex, even if the matrix $Q$ has only one negative eigenvalue \cite{Pardalos1991115}.
At present, only a few special subclass problems with QCQP can be determined as polynomial-time solvable \cite{Kim2003143,Zhang2000453,BurerSYY2020117,Azuma202282}.
Therefore, the widespread application of QCQP and the computational challenges it brings have attracted the attention of researchers in various disciplines \cite{Floud1995217,Pardalos1991115}, for which a lot of algorithms and theories have been reported in the literature.
Furthermore,   QCQP usually has multiple local optimal solutions, and traditional optimization methods can only find its local solutions.
However, many practical problems need to find their global optimal solutions or approximate global optimal solutions, so it is necessary to develop a global optimization algorithm for solving QCQP.
Branch-and-bound (B\&B) algorithm  \cite{Horst2000Chapter3} is a typical enumeration method that can solve QCQP  globally.
The most important part of this approach is a strategy that can estimate the lower (or upper) bound of each enumeration. As in \cite{Burer2008259,Vandenbussche2005559,JiaoH2015973,Bomze2002117}, various relaxations have been adopted to estimate the lower bound of each enumeration.
Among them, the semidefinite programming (SDP) relaxation has attracted widespread attention over the past few decades (see \cite{Bao2011129}  for a recent review) because of the polynomial time solvability \cite{Ye1997Interior} of SDP when the interior point method is adopted. In fact, under the Shor scheme \cite{Lemarechal2001119,Shor1987731}, the classical lagrangian relaxation is the same as SDP relaxation. Another widely used relaxation method is linear programming relaxation, such as reformulation linearization technique (RLT, see \cite{Sherali1990411}).
Besides,  a relatively new approach is to reconstruct the quadratic programming problem into a cone programming problem with copositive cone constraints \cite{Bomze2000301,Bomze2012509,Deng2013229} or set-semidefinite cone constraints \cite{Arima20132320,Lu20111475,Sturm2003246}, and then to approximate these cones by some computationally tractable cones.
However, a classical literature \cite{Chen20123352} shows that even for non-convex quadratic programming problems with linear constraints,
 these B\&B-based global solvers are capable of solving many small-scale problem instances, but may run out of time (or memory) and return only suboptimal solutions when the number of decision variables  ranges from 50 to 100.
 Inspired by such observations, Luo et al.\cite{Hezhi2018119} considered the large-scale QCQP with multiple negative eigenvalues, and treated its objective function as a special form of D.C. program. Then the objective function can be rewritten as:
$$
f(x)
=x^TQ_+x-x^TQ_-x+q^Tx
=x^TQ_+x+q^Tx-x^TC^TCx,
$$
 where $Q_+$ and $Q_-$ are positive semidefinite matrices obtained by the eigenvalue decomposition of  $Q$; $C\in \mathbb{R}^{r\times n}$ is a matrix whose every row $C_i\in \mathbb{R}^{n}$ is obtained by multiplying the square root of the absolute value of the $i$th negative eigenvalue of $Q$ by its corresponding eigenvector.
  Luo et al.  equivalently converted QCQP into the%
following optimization problem:
 \begin{equation}\label{eq1}
\min_{(x,t)\in \mathcal{X}}\left\{x^TQ_+x+q^Tx-\|Cx\|^2\right\}.
\end{equation}
 By adding bounds to the variables $t_i~(i=0,1,\cdots,r)$ in (\ref{eq1}), Zhang et al. \cite{ZBO2023} finally consider the following optimization
 problem:
 \begin{equation}\label{eqj2}
 \min_{y\in \mathcal{Y}}\left\{\nu(y):=y_0+\sum_{i=1}^r(-(y_i)^2)\right\},
\end{equation}
where $$\mathcal{Y}:=\{y\in \mathbb{R}^{r+1}:f_0(x)\leq y_0\leq y'_0,f_i(x)= y_i,i=1,2,\cdots,r ~for~some~x\in \mathcal{X}\}$$
with $y'_0\geq\max\limits_{x\in
\mathcal{X}}f_0(x)$, $f_0(x)=x^TQ_+x+q^Tx$, $f_i(x)=C_i^{\top}x$ for $i=1,2,\cdots,r$.
In view of (\ref{eqj2}), Zhang et al. developed a  rectangular B\&B algorithm, whose linear relaxation combines a new outer approximation technique and a special RLT.

In this paper, we continue to consider (\ref{eqj2}) and propose a new convex relaxation technique based on a simplex rather than a rectangle.

By projecting QCQP onto the result space, we obtain the equivalent problem of minimizing concave quadratic functions (all variables are in the result space) on a compact convex set. Then, based on the branching operation on the simplex, we propose a novel outer approximation technique and integrate it into the branch and bound framework to develop a new result space global optimization algorithm to solve the EP problem. In each iteration of the algorithm, only one convex quadratic programming problem and two small linear programming problems are solved, and only one new semi infinite linear inequality constraint is generated to minimize the computational complexity of the subproblems. Meanwhile, there always exists a decomposition between $n$  original decision variables and $r+1$ additional introduced variables, such that the number of variables for each subproblem solved by the algorithm is $n$ or $r+2$, rather than $n+r+1$.

The overall theory of the algorithm indicates that it is feasible, and when the number of negative eigenvalues r of matrix Q is very small, it is found that our algorithm may save more computation than the algorithm in \cite{ZBO2023} . Mainly due to the significantly fewer linear semi infinite constraints in the outer approximation space. Of course, our algorithm does not require parameter selection steps compared to the algorithm in \cite{ZBO2023}. This greatly saves the difficulty of adjusting parameters.

The rest of this article is structured as follows. In Section 2, an equivalent problem EP based on outer space and its novel linear relaxation problem were developed for problem QCQP. Section 3 is an analysis of the convergence of algorithm theory, which mainly includes three main operations (branching, bounding, pruning) and branch and bound algorithms. Finally, we summarize the paper in Section 4 and discuss some future research directions.


\section{Relaxation}
The main purpose of this section is to construct a simplex based relaxation strategy by utilizing (\ref{eqj2}). To this end, the related notions  are given first.

For an $r$-simplex $S$ with vertices $v^1,\cdots,v^{r+1}\in\mathbb{R}^m$, we denote the sets
of vertices and edges of $S$ by $V(S) = \{v^1,\cdots,v^{r+1}\}$ and $E(S) = \{\{v^i, v^j\} : i, j =1,\cdots, r+1,i<j\}$, respectively. Let $d(S) = \max\{\|v-w\|:\{v, w\}\in E(S)\}$ be the
diameter of $S$. $e_r=(1,\cdots, 1)^T\in \mathbb{R}^r$. The Euclidean norm of a vector $x\in \mathbb{R}^n$ is denoted by $\|x\|$. The standard $r$-simplex is denoted by
 $$\triangle_{r}=\left\{w\in \mathbb{ R}^{r+1}\bigg| \sum_{i=1}^{r+1}w_i=1,
w_i\geq0,i=1,\cdots,r+1\right\}.$$

The   global optimal solution $y^*$ of (\ref{eqj2}) clearly lie in the rectangle $\overline{\mathcal{H}}$, where
$$
\overline{\mathcal{H}}:=[\underline{y}^0,\overline{y}^0]=[\underline{y}_0, \overline{y}_0]\times\mathcal{H},~
 \mathcal{H}:=\prod\limits_{i=1}^r[\underline{y}_i, \overline{y}_i],~\underline{y}_0:=\min_{x\in \mathcal{X}}f_0(x),$$ $$ \overline{y}_0:=y_0',~\underline{y}_i:=\min_{x\in \mathcal{X}}f_i(x),~\overline{y}_i:=\max_{x\in \mathcal{X}}f_i(x),~i= 1,\cdots,r. $$
As described in  \cite{ZBO2023}, $\overline{\mathcal{H}}$ contains $\mathcal{Y}$ and also provides a rough initial outer approximation for $\mathcal{Y}$.

Let $S_0$ be a non-degenerate $r$-simplex cover of $ \mathcal{H}$. Then problem (\ref{eqj2}) can be equivalently converted into
 \begin{equation}\label{eqj3}
 \min_{y\in \mathcal{Y}\cap \bar{S}_0}\left\{\nu(y):=y_0+\sum_{i=1}^r(-(y_i)^2)\right\},
\end{equation}
where $\bar{S}_0=[\underline{y}_0, \overline{y}_0]\times S_0$.
 The problem (\ref{eqj3}) is obviously equivalent to the problem QCQP, and  $\{x\in \mathbb{R}^{n}:y^*=f(x),~x\in \mathcal{X}\}$ is a global  optimal solution set for QCQP (see Theorem 1 in  \cite{ZBO2023}).  Of course, it should be noted that (\ref{eqj3}) is a concave minimization problem with concave objective function and convex feasible region. It can be concluded that the optimal solution is located on the boundary of $\mathcal{Y}$, i.e.,  $y^*\in \partial \mathcal{Y}$.

Next, for any nondegenerate simplex $S\subseteq S_0$ with
$V(S)= \{v^1,\cdots,v^{r+1}\}$, $v^j=(y^j_1,y^j_2,\cdots,y^j_r)^{\top}\in \mathbb{R}^r$, $j=1,2,\cdots,r+1$, the subproblem of  problem  (\ref{eqj3}) over $S$ can be formalized as
 \begin{equation}\label{eqjz3}
 \min_{y\in \mathcal{Y}\cap \bar{S}}\left\{\nu(y):=y_0+\sum_{i=1}^r(-(y_i)^2)\right\},
\end{equation}
with $\bar{S}=[\underline{y}_0, \overline{y}_0]\times S$.
 Meanwhile, we consider the following convex optimization problem
 $$
({\rm CP}_{\lambda}):~
 \min\limits_{ x\in \mathcal{X}}~\left\{f_0(x)-2\sum_{i=1}^{r}\lambda_if_i(x)
 \right\},$$
where $\lambda:=(\lambda_1,\lambda_2,\cdots,\lambda_r)^{\top}
\in \mathbb{R}^{r}$ is the given parameter.
Let $x_{\lambda}$ and $\mu_{\lambda}$ be the optimal value and optimal solution of the problem ${\rm CP}_{\lambda}$, respectively, and let $y_{\lambda}:=(f_0(x_{\lambda}), \cdots,f_r(x_{\lambda}))^{\top}$,
then ${\rm CP}_{\lambda}$ has the following two properties in Theorems \ref{TH5}-\ref{TH6}.

\begin{theorem}\label{TH5}
 $y_{\lambda}\in \partial \mathcal{Y}$.
\end{theorem}
\begin{proof}
Since ${\rm CP}_{\lambda}$ is equivalent to the following convex optimization problem  with a linear objective function  with respect to  $y$:
\begin{equation}\label{ZJEQ1}
\min\limits_{y\in \mathcal{Y}}~\left\{y_0-2\sum_{i=1}^{r}\lambda_iy_i\right\},\end{equation}
 $y_{\lambda}$ is clearly  optimal for problem (\ref{ZJEQ1}). As a result, $y_{\lambda}$  must lie on the boundary of $\mathcal{Y}$, i.e., $y_{\lambda}\in \partial \mathcal{Y}$.   \qed
\end{proof}
\begin{theorem} \label{TH6}
For any $y\in \mathcal{Y}$, it holds that
  \begin{equation}\label{eq14}
\mu_{\lambda}\leq y_0-2\sum_{i=1}^{r}\lambda_iy_i.
\end{equation}
\end{theorem}
\begin{proof}
The proof of Theorem \ref{TH5} shows that $y_{\lambda}$ is  optimal for problem (\ref{ZJEQ1}), and then for any $y\in \mathcal{Y} $,   the inequality (\ref{eq14}) holds.    \qed
\end{proof}

According to Theorems \ref{TH5}-\ref{TH6}, the hyperplane
$$
\mathcal{T}_\lambda:=\left\{y\in \mathbb{R}^{r+1}:
  y_0-2\sum_{i=1}^{r}\lambda_iy_i-\mu_{\lambda}
=0\right\}
$$
employs $y_{\lambda}$ as its support point, which is a support plane of $\mathcal{Y}$.
Set the newly generated supporting semi-infinite space to
$$
\mathcal{T}_\lambda^-:=\left\{y\in \mathbb{R}^{r+1}:
y_0-2\sum_{i=1}^{r}\lambda_iy_i
\geq \mu_{\lambda}\right\},
$$
then $\mathcal{Y} \subset \mathcal{T}_\lambda^-$.
For each $j=1,2,\cdots,r+1$,  we obtain an outer approximation of $\mathcal{Y}$ with respect to the subsimplex $S$ by setting $\lambda=v^j$, which has the following form
$$
\mathcal{T}_S:=\left\{y\in \mathbb{R}^{r+1}:
y_0-2\sum_{i=1}^{r}v_i^jy_i
\geq \mu_{v^j},~j=1,2,\cdots,r+1\right\}.
$$
Further, let $yy=(y_1,y_2,\cdots,y_r)^{\top}$ and $g(yy)=-\sum_{i=1}^r(y_i)^2$, then  the convex envelope of the concave function $g(yy)$ over the simplex $S$ can be constructed as
 \begin{equation}\label{eqjz8}
\underline{g}(yy)=\sum_{j=1}^{r+1}w_jg(v^j),~yy=\sum_{j=1}^{r+1}w_jv^j ,~w\in \triangle_{r}.\end{equation}
After replacing $\mathcal{Y}$  with $\mathcal{T}_S$, removing the inequality constraint $ \underline{y}_0\leq{y}_0\leq \overline{y}_0 $, and employing Eq. (\ref{eqjz8}),  problem  (\ref{eqjz3}) can be relaxed into the following linear program:
 \begin{equation}\label{eqjz7}
 \begin{aligned}
LB(S):=&\min ~y_0+\sum_{j=1}^{r+1}w_jg(v^j)\\
&~{\rm s.t.}~~y_i=\sum_{j=1}^{r+1}w_jv^j_i ,i=1,2,\cdots,r,\\
&~~~~~~~w\in \triangle_{r},~y\in \mathcal{T}_S.
\end{aligned}
 \end{equation}
For each $i=1,2,\cdots,r$, if $y_i=\sum_{j=1}^{r+1}w_jv^j_i$ is brought into every linear semi-infinite constraint of  $\mathcal{T}_S$, it is easy to take the following linear programming problem with fewer variables:
 \begin{equation}\label{fjeqjz7}
 \begin{aligned}
LB(S):=&\min ~y_0+\sum_{j=1}^{r+1}w_jg(v^j)\\
&~{\rm s.t.}~~y_0-2\sum_{i=1}^{r}v_i^j\sum_{s=1}^{r+1}w_sv^s_i
\geq \mu_{v^j},~j=1,2,\cdots,r+1,\\
&~~~~~~~w\in \triangle_{r}.
\end{aligned}
 \end{equation}

 \begin{remark}If $S$ is degenerate, in the algorithm, let $\min\limits_{y\in \bar{S}\cap \mathcal{Y}}\nu(y)=LB(S)=+\infty$.  \end{remark}

%
%

Compared with the problem (\ref{eqj3}), problem (\ref{eqjz7}) not only relaxes the feasible region but also relaxes the objective function in form. Therefore, problem (\ref{eqjz7}) can provide an effective lower bound for the previous problem.
 In addition, the $y_{\lambda}$ constructed by the optimal solution $x_{\lambda}$ of the convex problem ${\rm CP}_{\lambda}$ is employed to update the upper bound on the global optimum of (\ref{eqj3}). In effect, $LB(S)$ is compact enough if the diameter $d(S)$ becomes sufficiently small, i.e.,
\begin{equation}\label{eqjq9} \lim_{d(S)\rightarrow0}\left(\min\limits_{y\in \bar{S}\cap \mathcal{Y}}\nu(y)-LB(S)\right)\rightarrow0,\end{equation}
 which will be proved in the subsequent Theorem \ref{jthe4}.

%

\section{Algorithm}

\subsection{Simplicial  bisection method}\label{ssub4.1}
Derived from the fact that the above convex relaxation is proposed based on the simplex, the branching operation of our algorithm will be performed on the simplex. Thus, the simplex $S_0$ will be subdivided into many sub-simplices.
For a given simplex $S\subseteq S_0$ with $V(S) = \{v^1,\cdots,v^{r+1}\}$ and $E(S) = \{\{v^i, v^j\} : i, j =1,\cdots, r+1,i<j\}$, we now give the branching rule as follows:

i) Choose a $\{v^{\tilde{i}}, v^{\tilde{j}}\}\in\arg\max\{\|v-w\|:\{v, w\}\in E(S)\}$ and set $\eta=\frac{v^{\tilde{i}}+v^{\tilde{j}}}{2}$;

ii) Divide $S$
into
$$
S_{1}=
\left\{\sum_{l=1,l\neq \tilde{j} }^{r+1}w_lv^l+w_{\tilde{j}}\eta\bigg|w\in \triangle_{r}\right\}~
{\rm and}~
 S_{2}=
\left\{\sum_{l=1,l\neq \tilde{i}}^{r+1}w_lv^l+w_{\tilde{i}}\eta \bigg|w\in \triangle_{r}\right\},
$$
where  $V(S_{1}) = \{v^1_{1},\cdots,v^{r+1}_{1}\}$ with $v^{\tilde{j}}_{1}=\eta$,
$v^l_{1}=v^l$ for $l\in\{1,\cdots,r+1\}\backslash\{\tilde{j}\}$,  $V(S_{2}) = \{v^1_{2},\cdots,v^{r+1}_{2}\}$ with $v^{\tilde{i}}_{2}=\eta$,
$v^l_{2}=v^l$ for $l\in\{1,\cdots,r+1\}\backslash\{\tilde{i}\}$.


From the above simplex partition rule, we know that $S_{1}\cap S_{2}={\rm rbd}(S_{1})\cap {\rm rbd}(S_{2})$, $S_{1}\cup S_{2}=S$ and $S_{s}\subset S,s=1,2$, where ${\rm rbd}(S_{1})$ denotes the relative boundary
of  $S_{1}$.
Besides, this simplex partition rule follows the longest edge partitioning rule, so it generates an exhaustive nested partition of a given simplex in the limit  sense \cite{Dickinsonhapter3}. Thus, the global convergence of the B\&B algorithm is guaranteed.
\subsection{Simplicial branch-and-bound algorithm}
\begin{algorithm}[h!]
\textbf{Algorithm}~(OSSBBA)\\
\textbf{Step 0 (Initialization).}

Given a  tolerance $\epsilon>0$. Calculate $\underline{y}_i:=\min\limits_{x\in \mathcal{X}}f_i(x)$ and $\overline{y}_i:=\max\limits_{x\in \mathcal{X}}f_i(x)$, $i=1,\cdots,r$  for  (\ref{eqj2}).

 Initialize the simplex $S_0$ with a set  $V(S_0)=\{v^1_{0},\cdots,v^{r+1}_{0}\}$ of vertices, where $v^1_{0}=\underline{yy}:=(\underline{y}_1,\cdots,\underline{y}_r)^{\top}$, $v^{i+1}_{0}=(\underline{y}_1,\cdots,\underline{y}_i+r(\overline{y}_i-\underline{y}_i),\cdots,\underline{y}_r)$, $i=1,2,\cdots,r$.
  Clearly, $S_0$ covers $\mathcal{H}$.

For each $j=1,2,\cdots,r+1$, solve the convex programming problem ${\rm CP}_{\lambda}$ with $\lambda=v^j_{0}$, and obtain the optimal solution $x_{v^j_{0}}$ and the optimal value $\mu_{v^j_{0}}$.  Set $\mu_{S_0}:=(\mu_{v^1_{0}},\mu_{v^2_{0}},\cdots,\mu_{v^{r+1}_{0}})$

Compute $x^*=\arg\min\{\phi(x_{v^j_{0}}):i=1,2,\cdots,r+1\}$, $y^*:=(f_0(x_{*}), \cdots,f_r(x_{*}))^{\top}$ and let $UB^0=\nu(y^*)=\phi(x^*)$  be the initial upper bound on the optimal value of problem (\ref{eqj3}).

Solve the linear program (\ref{fjeqjz7}) to obtain an initial lower bound $LB^0=LB(S_0)$ on the optimal value of (\ref{eqj3}).
 Set $\Xi=\{[S_0,\mu_{S_0},E(S_0),LB(S_0)]\}$, $\epsilon>0$,
   $k:=0$.\\
\textbf{Step 1 (Termination).}\\
If $UB^k-LB^k\leq \epsilon$, terminate and output $x^*$ and $y^*$.\\
\textbf{Step 2. (Simplicial subdivision).}

Choose a $\{v^{\tilde{i}}_k, v^{\tilde{j}}_k\}\in\arg\max\{\|v-w\|:\{v, w\}\in E(S_k)\}$ and set $\eta=\frac{v^{\tilde{i}}_k+v^{\tilde{j}}_k}{2}$.

By using the simplicial bisection method shown in Sect. \ref{ssub4.1}, the simplex $S_k$ is divided into two sub-simplices $S_{k1}$ and $S_{k2}$ with $V(S_{k1})=\{v^1_{k1},\cdots,v^{r+1}_{k1}\}$  and $V(S_{k2})=\{v^1_{k2},\cdots,v^{r+1}_{k2}\}$, where
 $v^{\tilde{j}}_{k1}=\eta$, $v^j_{k1}=v^j_{k}$, $j\in\{1,2,\cdots,r+1\}\backslash\{\tilde{j}\}$,
$v^{\tilde{i}}_{k2}=\eta$, $v^i_{k2}=t^i_{k}$, $i\in\{1,2,\cdots,r+1\}\backslash\{\tilde{i}\}$.
\\
\textbf{Step 3. (Determine the upper bound).}

Solve the convex  problem ${\rm CP}_{\lambda}$ with $\lambda=\eta$  and obtain the optimal solution $x_{\eta}$ and the optimal value $\mu_{\eta}$.

Set $\mu_{S_{k1}}=(\mu_{v^1_{k1}},\mu_{v^2_{k1}},\cdots,\mu_{v^{r+1}_{k1}})$ with
$\mu_{v^{\tilde{j}}_{k1}}=\mu_{\eta}$, $\mu_{v^j_{k1}}=\mu_{v^j_{k}}$, $j\in\{1,2,\cdots,r+1\}\backslash\{\tilde{j}\}$, $\mu_{S_{k2}}=(\mu_{v^1_{k2}},\mu_{v^2_{k2}},\cdots,\mu_{v^{r+1}_{k2}})$ with $\mu_{v^{\tilde{i}}_{k2}}=\mu_{\eta}$, $\mu_{v^i_{k2}}=\mu_{v^i_{k}}$,  $i\in\{1,2,\cdots,r+1\}\backslash\{\tilde{i}\}$.

Set $y_\eta=(f_0(x_{\eta}), \cdots,f_r(x_{\eta}))^{\top}$, $U^k=\nu(\eta)$, if $U^k<UB^k$, set $UB^k=U^k$, $y^*=y_\eta$, $x^*=x_\eta$.\\
\textbf{Step 4. (Simplicial delete).}

Compute $LB(S_{k1})$ and $LB(S_{k2})$ by addressing  the linear program (\ref{fjeqjz7}) with $S=S_{k1},S_{k2}$.

Set~$\Xi:=\Xi\backslash \{[S_k,\mu_{S_k},E(S_k),LB(S_k)]\} \cup \{[S_{k1},\mu_{S_{k1}},E(S_{k1}),LB(S_{k1})],[S_{k2},\mu_{S_{k2}}$, $E(S_{k2}),LB(S_{k2})]\}$.

For each $s=1,2$, if $UB^k-LB(S_{ks})\leq\epsilon$, set $\Xi:=\Xi\backslash  \{[S_{ks},\mu_{S_{ks}}$, $E(S_{ks}),LB(S_{ks})]\}$.\\
\textbf{Step 5. (Determine the lower bound).}

If $\Xi\neq\emptyset$,
set $LB^{k}:=\min\{LB(S):[\cdot,\cdot,\cdot,LB(S)]\in\Xi\}
$ and goto Step 6, otherwise stop and output $t^*$, $k=k+1$.\\
\textbf{Step 6 (Select the simplex).}

Choose an element $\{[S_{k},\mu_{S_{k}}$, $E(S_{k}),LB(S_{k})]\}\in  \Xi$ such that  $LB(S_{k})=LB^k$.

Set $k:=k+1$ and return to Step 1.
\end{algorithm}

In the above algorithm, $k$ is adopted as the iteration index. At each iteration, one subproblem is selected and up to two new subproblems are created to replace the old one. The optimal value of the new subproblem will gradually improve compared with the old subproblem, and the corresponding upper and lower bounds will be updated, so that the gap between the upper and lower bounds will gradually decrease.
Indeed, for a given tolerance $\epsilon>0$, OSSBBA can output a global $\epsilon$-optimal solution for problem (\ref{eqj3}). Here, we call $t^*\in S_0$ a global $\epsilon$-optimal solution to problem (\ref{eqj3})  if $\nu(y^*)\leq \min\limits_{y\in \mathcal{Y}}\nu(y)+\epsilon$.

\begin{theorem}\label{jthe4}
Given a tolerance $\epsilon>0$, when {\rm OSSBBA} runs to Step 1 of the $k$th iteration, if the subproblem $\{[S_k,\mu_{S_{k}},E(S_k),LB(S_k)]\}$ satisfies $d(S)\leq \epsilon/(4d_S)$ with $d_S=\max\{\|v_k^i\||i=1,\cdots,r+1\}$, the algorithm must terminate and output a global $\epsilon$-optimal solution to problem (\ref{eqj3}). Besides, $ \min\limits_{y\in S}\nu(y)-LB(S_k)\leq\epsilon.$
 \end{theorem}
\begin{proof} For convenience, after ignoring the index $k$, let $S=S_k$, $LB=LB^k$, $UB=UB^k$ and $v^j=v_k^j$, $j=1,\cdots,r+1$.
 Let  $\tilde{\nu} =\min\{y_{v^j_0}-\|y_{v^j}\|^2:j=1,\cdots,r+1\}$,
 and then from (\ref{eqjz7}) we have
\begin{equation}\label{eqj7}  \begin{aligned}
LB(S)&= \tilde{y}_0+\sum_{j=1}^{r+1}\tilde{w}_jg(v^j)\\
&\geq \max_{1\leq j\leq r+1} \left\{2\sum_{i=1}^{r}v_i^j\tilde{y}_i
+\mu_{v^j}\right\} -\sum_{j=1}^{r+1}\tilde{w}_j \|v^j \|^2\\
&\geq \sum_{j=1}^{r+1}\tilde{w}_j\left(2\sum_{i=1}^{r}v_i^j\tilde{y}_i
+\mu_{v^j}\right) -\sum_{j=1}^{r+1}\tilde{w}_j \|v^j \|^2\\
&= \sum_{j=1}^{r+1}\tilde{w}_j\left(2\sum_{i=1}^{r}v_i^j\tilde{y}_i
+y_{v^j_0}-2\sum_{i=1}^{r}v_i^jy_{v^j_i}- \|v^j \|^2\right)\\
&= \sum_{j=1}^{r+1}\tilde{w}_j\left(y_{v^j_0}-\|y_{v^j}\|^2+2\sum_{i=1}^{r}v_i^j(\tilde{y}_i-v^j_i)
+ \|y_{v^j}- v^j \|^2\right)\\
&\geq \sum_{j=1}^{r+1}w_j\left(y_{v^j_0}-\|y_{v^j}\|^2+2\sum_{i=1}^{r}v_i^j(\tilde{y}_i-v^j_i)\right)\\
&\geq\tilde{\nu} +2\sum_{j=1}^{r+1}\tilde{w}_j\left(\left\|\widetilde{yy}\right\|^2-\|v^j\|^2\right),
\end{aligned}\end{equation}
 where  $\widetilde{yy}=(\tilde{y}_1,\tilde{y}_2,\cdots,\tilde{y}_r)\in S$  and $(\tilde{y}_0,\tilde{yy},\tilde{w})$ is an optimal solution of the  problem (\ref{eqjz7}).
Besides, for all $yy=(y_1,y_2,\cdots,y_r)\in S$, it holds  that
\begin{equation}\label{eqj9}  \|v^j\|^2-\left\|yy\right\|^2  =(\|v^j\|+\|yy\|)(\|v^j\|-\|yy\|)\leq2d_S(\|v^j-yy\|)\leq 2d_Sd(S).\end{equation}
Also, since $LB=LB(S)$ is the least lower bound at the current iteration, we have $\tilde{\nu}\geq UB\geq \min\limits_{y\in \bar{S}\cap\mathcal{Y}}\nu(y)\geq LB=LB(S)$, where  $\bar{S}=[\underline{y}_0, \overline{y}_0]\times S$.
Thus, it follows from Eqs.  (\ref{eqj7})-(\ref{eqj9}) that
\begin{equation}\label{eqj1410}\min\limits_{y\in \bar{S}\cap\mathcal{Y} }\nu(y)-LB(S)\leq UB-LB\leq 2\sum_{j=1}^{r+1}\tilde{w}_j\left(\|v^j\|^2-\left\|\widetilde{yy}\right\|^2\right)\leq 4d_Sd(S).\end{equation}
And when $d(S)\leq \epsilon/(4d_S)$ is satisfied, we have
\begin{equation}\label{eqj110} \min\limits_{y\in \bar{S}\cap\mathcal{Y} }\nu(y)-LB(S)\leq UB-LB\leq\epsilon.\end{equation}
Besides,  it knows from the iterative mechanism of OSSBBA that $LB\leq\min\limits_{y\in \mathcal{Y}}\nu(y)\leq UB=\nu(y^*)$.
Then by combining Eq. (\ref{eqj110}), it follows that
$$\nu(y^*)= UB \leq  LB +\epsilon\leq \min\limits_{y\in \mathcal{Y}}\nu(y)+\epsilon,$$ which means that $y^*$ is a global $\epsilon$-optimal solution to (\ref{eqj3}).
\end{proof}
\begin{remark}\label{REK1}From Theorem \ref{jthe4},  we know that
Eq. (\ref{eqj1410}) implies that Eq. (\ref{eqjq9}) holds.\end{remark}

 Theorem \ref{jthe4} and Remark \ref{REK1} show  that the optimal value of problem (\ref{eqj3}) over each subsimplex and that of its relaxation problem (\ref{eqjz7}) are gradually approaching in the limiting sense. This  implies that the bounding and branching operations are consistent, so the B\&B algorithm is theoretically globally convergent.


Now, let us analyze the complexity of OSSBBA based on Theorem \ref{jthe4} and the iterative mechanism of this algorithm.
\begin{lemma} \label{leme5}
 For any $yy=({y}_1,{y}_2,\cdots,{y}_r)\in S_0$  with  $V(S_0)=\{v^1_{0},\cdots,v^{r+1}_{0}\}$, it holds that \begin{equation}\label{jdEQ01q}\|yy\|\leq\bar{d}:=\max\left\{\sqrt{ \sum_{i=1,i\neq s}^r(\underline{y}_{i})^2+
(\underline{y}_{s}+r(\overline{y}_{s}-\underline{y}_{s}))^2}:i=1,\cdots,r\right\},\end{equation}
where  $\overline{y}_i$ and $\underline{y}_i$ is given in {\rm Step 0} of {\rm OSSBBA}.
\end{lemma}
\begin{proof}
By considering the convex maximization problem $\max\limits_{yy\in S_0}\|yy\|$, we know that its optimal solution lies at a certain vertex of $S_0$. It follows that
\begin{equation}\label{jdEQ1q}\|v_0^{1}\|\leq\|yy\|\leq \max\limits_{yy\in S_0}\|yy\|=\max\{\|v_0^{s+1}\|:s=1,\cdots,r\}.\end{equation}
For each $s=1,\cdots,r$, it can be derived that
$$\|v_0^{s+1}\|^2=\sum_{i=1,i\neq s}^r(\underline{y}_{i})^2+
(\underline{y}_{s}+r(\overline{y}_{s}-\underline{y}_{s}))^2.
$$
By combining Eq. (\ref{jdEQ1q}) it concludes that Eq. (\ref{jdEQ01q}) is true. This achieves the proof.

\end{proof}

\begin{theorem} \label{jthe5}
Given a tolerance $\epsilon>0$, the maximum number of iterations required by {\rm OSSBBA} to obtain a global $\epsilon$-optimal solution for {\rm QCQP} is
$$\left\lfloor\frac{\prod_{i=1}^r(\overline{y}_i-\underline{y}_i)}{\sqrt{r+1}}\left(\frac{8\sqrt{2}r\bar{d}}{\epsilon}\right)^r\right\rfloor,$$
where
$\bar{d}=\max\left\{\sqrt{ \sum_{i=1,i\neq s}^r(\underline{y}_{i})^2+
(\underline{y}_{s}+r(\overline{y}_{s}-\underline{y}_{s}))^2}:i=1,\cdots,r\right\}.$
\end{theorem}
\begin{proof}
When OSSBBA terminates, either $k=0$ or $k\geq1$. If $k=0$, the algorithm does not enter the iteration loop.
Thus,
let us talk about the case where the algorithm terminates after many iterations.

At the case of $k\geq1$, it follows from Theorem \ref{jthe4} that $d(S_k)\leq \epsilon/(4d_{S_k})$ is essentially a  sufficient condition for the termination criterion of the algorithm, $UB^k-LB^k\leq\epsilon$, to hold.
Since $v_k^i\in S_k\subseteq S_0$ for $i=1,\cdots,r+1$, it follows from Lemma \ref{leme5}  that $d_{S_k}=\max\limits_{1\leq i\leq r+1}\|v_k^i\| \leq \bar{d}$, which means $d(S_k)\leq \epsilon/(4\bar{d})$ is also a sufficient condition for $UB^k-LB^k\leq\epsilon$.
Further, according to the simplicial branching rule in Step 2, a total of $k+1$ sub-simplices are generated for the initial simplex $S_0$.
For convenience, we denote these subsimplices as $S_1,S_2,\cdots,S_{k+1}$, respectively.
Obviously, $S_0=\bigcup\limits_{t=1}^{k+1}S_t$.
Also, let $V(S_t) = \{v_t^1,\cdots,v_t^{r+1}\}$, $E(S_t) = \{\{v_t^i, v_t^j\} : i, j =1,\cdots, r+1,i<j\}$ for $t=1,\cdots, k+1$.
 In the worst case, suppose that the longest edge of each subsimplex $S_t$ satisfies $d(S_t)\leq \epsilon/(4\bar{d})$. At this point, every edge $\{v_t^i, v_t^j\}$ of $S_t$ satisfies
\begin{equation}\label{EQ12}\|v_t^i-v_t^j\| \leq \epsilon/(4\bar{d}),~i, j =1,\cdots, r+1,i<j,t=1,\cdots, k+1,\end{equation}
which implies that the volume $Vol(S_t)$ of $S_t$ does not exceed the volume $Vol(\bar{S})$ of a $r$-simplex $\bar{S}$  with a unique edge length $\epsilon/(4\bar{d})$.
Thus, we have
\begin{equation}\label{EQ13}Vol(S_0)=\sum\limits_{t=1}^{k+1}Vol(S_t)\leq (k+1)Vol(\bar{S}).\end{equation}
By Cayley-Menger Determinant, it can be calculated that
\begin{equation}\label{EQ14}\begin{aligned}&Vol(S_0)=\sqrt{\frac{(-1)^{r+1}}{2^r(r!)^2}
\det\left(\left|\begin{matrix}
0 & e_{r+1}^T \\
e_{r+1} & \hat{B} \\
\end{matrix}\right|\right)}=\frac{r^r\prod_{i=1}^r(\overline{y}_i-\underline{y}_i)}{r!},\\ &Vol(\bar{S})=\sqrt{\frac{(-1)^{r+1}}{2^r(r!)^2}\det\left(\left|\begin{matrix}
0 & e_{r+1}^T \\
e_{r+1} & \check{B} \\
\end{matrix}\right|\right)}=
\frac{\sqrt{r+1}}{r!}\left(\frac{\epsilon}{8\sqrt{2}\bar{d}}\right)^r,\end{aligned}\end{equation} where   $\hat{B}=(\hat{\beta}_{ij})$ and $\check{B}=(\check{\beta}_{ij})$ denote two $(r+1)\times(r+1)$ matrixes given by $\hat{\beta}_{ij}=\|v_0^i-v_0^j\|^2$, $\check{\beta}_{ij}=\left(\epsilon/(4\bar{d})\right)^2$ for
$i, j =1,\cdots, r+1, i\neq j$, and $\hat{\beta}_{ii}=\check{\beta}_{ii}=0$ for $i =1,\cdots, r+1$. From the properties of the vertices of $S_0$ given in Step 0 of OSSBBA, it follows that $\hat{\beta}_{1j}=\hat{\beta}_{j1}=r^2(\overline{y}_{j-1}-\underline{y}_{j-1})^2$ for $j=2,3,\cdots, r+1$ and $\hat{\beta}_{ij}=\hat{\beta}_{ji}=r^2[(\overline{y}_{i-1}-\underline{y}_{i-1})^2+(\overline{y}_{j-1}-\underline{y}_{j-1})^2]$ for $i, j =2,3,\cdots, r+1, i\neq j$. Next, by combining Eqs. (\ref{EQ13}) and (\ref{EQ14}), we have
$$k\geq \frac{Vol(S_0)}{Vol(\bar{S})}-1=\frac{\prod_{i=1}^r(\overline{y}_i-\underline{y}_i)}{\sqrt{r+1}}\left(\frac{8\sqrt{2}r\bar{d}}{\epsilon}\right)^r-1.$$
However, when Eq. (\ref{EQ12}) holds and $k=\left\lfloor\frac{\prod_{i=1}^r(\overline{y}_i-\underline{y}_i)}{\sqrt{r+1}}\left(\frac{8\sqrt{2}r\bar{d}}{\epsilon}\right)^r\right\rfloor$, these $k+1$ subsimplices must be deleted in Step 4.
Thus, the number of iterations at which OSSBBA terminates is at most $\left\lfloor\frac{\prod_{i=1}^r(\overline{y}_i-\underline{y}_i)}{\sqrt{r+1}}\left(\frac{8\sqrt{2}r\bar{d}}{\epsilon}\right)^r\right\rfloor$. This completes the proof.

 \end{proof}
\begin{remark}
Theorem 3 reveals that when OSSBBA finds a global $\epsilon$-optimal solution for QCQP, the computational time required is at most $$(T_n+2T_r)\left\lfloor\frac{\prod_{i=1}^r(\overline{y}_i-\underline{y}_i)}{\sqrt{r+1}}\left(\frac{8\sqrt{2}r\bar{d}}{\epsilon}\right)^r\right\rfloor+(r+1)T_n$$  seconds, where $T_n$ and $T_r$ denote the upper bounds of the time required to solve a convex quadratic programming problem ${\rm CP}_{\lambda}$ and solve a linear programming problem (\ref{fjeqjz7}), respectively.
\end{remark}
\begin{remark}
Theorem 3 sufficiently guarantees that the algorithm OSSBBA completes termination in a finite number of iterations because of the existence of this most extreme number of iterations.
\end{remark}

\section{Conclusions}
In this paper, we have considered the problem QCQP with a few negative eigenvalues, which is arised from various disciplines, and is known as NP-hard. By embedding the modern outer approximation technique, linear relaxation, and initialization into the branch-and-bound framework, we developed the  OSSBBA algorithm to find the globally optimal solution for QCQP. From the sense of limit, we established the global convergence of OSSBBA.
From the complexity theory, we also noticed that as the number of negative eigenvalues in the Hessian matrix of the objective function increases, the computational cost of OSSBBA shows a rapid growth trend. In addition, compared with the algorithm in \cite{ZBO2023}, OSSBBA does not require parameter adjustment and does not require saving a large number of outer approximation planes. One interesting direction for future research is  to investigate whether we can develop an efficient global algorithm for general quadratic programming problems based on other outer approximation and branch-and-bound techniques.
\begin{acknowledgements}
This research is supported by the National Natural Science Foundation of China under Grant (12301401,11961001),  the Construction Project of first-class subjects in Ningxia higher Education(NXYLXK2017B09) and the Major proprietary funded project of North Minzu University(ZDZX201901).
\end{acknowledgements}

%
%



\end{document}